\documentclass[preprint,12pt]{elsarticle}
\usepackage{latexsym}
\usepackage{amssymb}
\usepackage{amsmath}
\usepackage{amsthm}
\usepackage{verbatim}
 \usepackage[pagewise]{lineno}
\numberwithin{equation}{section}
\newtheorem{theorem}{Theorem}[section]
\newtheorem{lemma}{Lemma}[section]

\newtheorem{definition}{Definition}[section]
\newcommand{\sign}[1]{\mathrm{sgn}(#1)}

\allowdisplaybreaks
\usepackage{geometry}
\biboptions{sort&compress}

\begin{document}
\UseRawInputEncoding 
\begin{frontmatter}



\title{Global conservative weak solutions for a class of nonlinear
dispersive wave equations beyond wave breaking
}

\author[ad1,ad2]{Yonghui Zhou}
\ead{zhouyh318@nenu.edu.cn}
\author[ad1]{Shuguan Ji\corref{cor}}
\ead{jisg100@nenu.edu.cn}
\address[ad1]{School of Mathematics and Statistics and Center for Mathematics and Interdisciplinary Sciences, Northeast Normal University, Changchun 130024, P.R. China}
\address[ad2]{School of Mathematics and Statistics, Hexi University, Zhangye 734000, P.R. China}
\cortext[cor]{Corresponding author.}

\begin{abstract}

In this paper, we study the global conservative weak solutions for a class of nonlinear dispersive wave equations after wave breaking. We first transform the equations into an equivalent semi-linear system by introducing new variables. We then establish the global existence of solutions for the semi-linear system by using the standard theory of ordinary differential equations. Finally, returning to the original variables, we obtain the global conservative weak solutions for the original equations.

\end{abstract}


\end{frontmatter}


\section{Introduction}
\label{sec:1}

In this paper, we consider the Cauchy problem for a class of nonlinear
dispersive wave equations with the following form
\begin{equation}
u_{t}-u_{txx}+(f(u))_{x}-(f(u))_{xxx}+\left(g(u)+\frac{f''(u)}{2}u_{x}^{2}\right)_{x}=0
\label{101}
\end{equation}
subject to the initial datum
\begin{equation}
u(0,x)=u_{0}(x),
\label{102}
\end{equation}
which was proposed by Holden and Raynaud \cite{Holden2007}, where $f(u), g(u)\in C^{\infty}(\mathbb{R},\mathbb{R})$ satisfy $g(0)=0$. Subsequently, Tian, Yan and Zhang \cite{Tian2014} investigated the local well-posedness of solutions for the Cauchy problem \eqref{101}--\eqref{102}. Novruzov \cite{Novruzov2017} established a local-in-space blowup criterion of solutions for the Cauchy problem \eqref{101}--\eqref{102}.

For $f(u)=\frac{u^{2}}{2}$ and $g(u)=u^{2}$, \eqref{101} corresponds to the classical Camassa-Holm equation
\begin{equation}
u_{t}-u_{txx}+3uu_{x}=2u_{x}u_{xx}+uu_{xxx},
\label{103}
\end{equation}
which is a well-known mathematical model describing the unidirectional propagation of shallow water waves, where $u(t,x)$ represents the fluid's free surface above a flat bottom. Such model was derived physically by Camassa and Holm \cite{Camassa1993} in 1993. Since then, its mathematical properties have been investigated extensively, such as
bi-Hamiltonian structure \cite{Fokas1981}, complete integrability \cite{Camassa1993,Constantin2001}, local well-posedness  \cite{Constantin19971,Constantin1,Constantin19981,Li2000}, wave breaking phenomena \cite{Constantin19971,Constantin1,Constantin1998,Constantin19981,Constantin19982,Constantin19983,
Constantin2000,Constantin20002,Li2000}, global existence of strong solutions \cite{Constantin19971,Constantin19981,Constantin19983} and global existence of weak solutions \cite{Constantin19984,Constantin20004,Holden20071,Holden2008,Wahlen2006,Xin2000}.

For $f(u)=\frac{k}{2}u^{2}$ and $g(u)=\frac{3-k}{2}u^{2}$, \eqref{101} corresponds to the hyper-elastic rod wave equation
\begin{equation}
u_{t}-u_{txx}+3uu_{x}=k\left(2u_{x}u_{xx}+uu_{xxx}\right),
\label{104}
\end{equation}
which was introduced by Dai \cite{Dai1,Dai2}, and describes far-field, finite length, finite amplitude radial deformation waves in cylindrical compressible hyper-elastic rods and $u$ represents the radial stretch relative to a pre-stressed state. Also notice that when parameter $k=1$, then \eqref{104} is reduced to \eqref{103}.
The local well-posedness, wave breaking phenomena and global existence of strong solutions of the Cauchy problem for \eqref{104} has been investigated in \cite{Brandolese20141,Brandolese20142,Yin2004,Zhou2005}.

For $f(u)=\frac{k}{2}u^{2}$, \eqref{101} corresponds to the generalized hyper-elastic rod wave equation
\begin{equation}
u_{t}-u_{txx}+\left(g(u)\right)_{x}+k uu_{x}=k\left(2u_{x}u_{xx}+uu_{xxx}\right),
\label{105}
\end{equation}
which was firstly studied by Coclite, Holden and Karlsen \cite{Coclite20051} in 2005.
Thereafter, many authors investigated the mathematical properties of problem \eqref{105}, such as well-posedness \cite{Mustafa20071}, wave breaking phenomena \cite{Mustafa20071} and global weak solutions \cite{Coclite20052}.

In the last 30 years, the Camassa-Holm equation and its various generalizations
were studied due to its many very interesting and remarkable properties, see \cite{Brandolese2014,Brandolese20141,Chen2011,Escher2007,Fu2010,Gui2010,Ji2021,Ji2022,Novruzov2022,Zhou2022}. However, their works are mainly to consider the behavior of solutions before and during the occurrence of wave breaking phenomena. In view of the possible development of singularities in finite time, it is natural to wonder about the behavior of solutions after the occurrence of wave breaking phenomena.
In 2007, by the characteristic method, Bressan and Constantin firstly proved that solutions of the Camassa-Holm equation \eqref{103} can be continued as either global conservative weak solution \cite{Bressan2007} or global dissipative weak solutions \cite{Bressan20071}. Afterwards, Mustafa \cite{Mustafa20072} obtained the global conservative weak solutions for the hyper-elastic rod wave equation \eqref{104} by using the same method as \cite{Bressan2007}. In 2015, Bressan, Chen and Zhang \cite{Bressan2015} investigated the uniqueness of the conservative solution for the Camassa-Holm equation \eqref{103}. Chen, Chen and Liu \cite{Chen2018} investigated the existence and uniqueness of global conservative weak solutions for the Novikov equation. Tu, Liu and Mu \cite{Tu2019} obtained the existence and uniqueness of the global conservative weak solutions for the rotation-Camassa-Holm equation.

Inspired by the previous work, in this paper, we study the global conservative weak solutions of the Cauchy problem \eqref{101}--\eqref{102}.

The rest of this paper is organized as follows. In Section \ref{sec:2}, we give the energy conservation law and some basic estimates. In Section \ref{sec:3}, we introduce a new set of independent and dependent variables, and transform the equation \eqref{101} into an equivalent semi-linear system under new variables. In Section \ref{sec:4}, we establish the global existence of solutions for the semi-linear system. In Section \ref{sec:5}, by inverse transformation method, we prove the existence of the global conservative weak solution for equation \eqref{101}.

\section{Preliminary}

\setcounter{equation}{0}

\label{sec:2}

In this section, we consider the following nonlocal form of a class of nonlinear
dispersive wave equations
\begin{equation}
u_{t}+f'(u)u_{x}+P_{x}=0
\label{201}
\end{equation}
equivalent to \eqref{101}, where the source term $P$ is defined by
\begin{equation}
P=p\ast\left(g(u)+\frac{f''(u)}{2}u_{x}^{2}\right)
=\frac{1}{2}e^{-|x|}\ast\left(g(u)+\frac{f''(u)}{2}u_{x}^{2}\right).
\label{202}
\end{equation}

For smooth solutions, differentiating \eqref{201} with respect to $x$ and using the relation $p_{xx}\ast h=p\ast h-h$, we get
\begin{equation}
u_{tx}+\frac{f''(u)}{2}u_{x}^{2}+f'(u)u_{xx}+P-g(u)=0.
\label{203}
\end{equation}
Multiplying \eqref{201} by $2u$ and \eqref{203} by $2u_{x}$, and adding the two resulting equations, we can get the following equation
\begin{equation}
(u^{2}+u_{x}^{2})_{t}+\left(f'(u)(u^{2}+u_{x}^{2})\right)_{x}
=-2(uP)_{x}+\left(2g(u)+f''(u)u^{2}\right)u_{x}.
\label{204}
\end{equation}
Define
\begin{equation}
H(u)=\int_{0}^{u}\left(2g(s)+f''(s)s^{2}\right)ds.
\label{205}
\end{equation}
Then \eqref{204} can be rewritten as
\begin{equation}
(u^{2}+u_{x}^{2})_{t}+\left(f'(u)(u^{2}+u_{x}^{2})\right)_{x}
=\left(H(u)-2uP\right)_{x}.
\label{206}
\end{equation}
Integrating \eqref{206} with respect to $t$ and $x$ over $[0,t]\times \mathbb{R}$, we get
\begin{equation}
E(t)=\int_{\mathbb{R}}(u^{2}(t,x)+u_{x}^{2}(t,x))dx=E(0).
\label{207}
\end{equation}

Since $f(u), g(u)\in C^{\infty}(\mathbb{R},\mathbb{R})$ and $g(0)=0$, then we have
\begin{equation}
|g(u(x))|\leq \sup_{|s|\leq\|u\|_{L^{\infty}}}|g'(s)||u(x)|\leq C(\|u\|_{1})|u(x)|.
\label{208}
\end{equation}
On the other hand, we can also easily verify that $|f''(u)|<C_{1}$ for some positive constant $C_{1}$. Therefore, we get
\begin{align}
\|P\|_{L^{2}}
\leq&\frac{1}{2}\left\|e^{-|x|}\right\|_{L^{1}}\|g(u)\|_{L^{2}}+ \frac{1}{2}\left\|e^{-|x|}\right\|_{L^{2}}\left\|\frac{f''(u)}{2}u_{x}^{2}\right\|_{L^{1}}\nonumber\\
\leq& C\left(\|g(u)\|_{L^{2}} +\|u\|_{1}^{2}\right)\nonumber\\
\leq& CE(0).
\label{209}
\end{align}
Similarly, we can obtain
\begin{equation}
\|P_{x}\|_{L^{2}}, \|P\|_{L^{\infty}},\|P_{x}\|_{L^{\infty}}\leq CE(0).
\label{2010}
\end{equation}

\section{Semi-linear system for smooth solutions}

\setcounter{equation}{0}

\label{sec:3}

In this section, we establish a semi-linear system for smooth solutions. To this end, it is essential to introduce the characteristic equation
\begin{equation}
\frac{dx(t)}{dt}=f'(u(t,x)).
\label{301}
\end{equation}
For any fixed point, the characteristic curve crossing the point $(t,x)$ is defined by setting
\begin{equation}
\gamma\rightarrow x^{c}(\gamma;t,x).
\label{302}
\end{equation}

In what follows, we use the energy density $(1+u_{x}^{2}(0,\bar{x}))$ to define the characteristic coordinate $Z=Z(t,x)$,
\begin{equation}
Z(t,x):=\int_{0}^{x^{c}(0;t,x)}(1+u_{x}^{2}(0,\bar{x}))d\bar{x}.
\label{303}
\end{equation}
Therefore, $Z(t,x)$ satisfies
\begin{equation}
Z_{t}+f'(u)Z_{x}=0, \ \ (t,x)\in \mathbb{R_{+}}\times\mathbb{R}.
\label{304}
\end{equation}
We also define $T=t$ to obtain the new coordinate $(T,Z)$. Then for any smooth function $h(T,Z)=h(t,Z(t,x))$, by \eqref{304}, we get
\begin{align}
h_{t}+f'(u)h_{x}
&=h_{T}T_{t}+h_{Z}Z_{t}+f'(u)\left(h_{T}T_{x}+h_{Z}Z_{x}\right)\nonumber\\
&=h_{T}\left(T_{t}+f'(u)T_{x}\right)+h_{Z}\left(Z_{t}+f'(u)Z_{x}\right)\nonumber\\
&=h_{T}
\label{305}
\end{align}
and
\begin{equation}
h_{x}=h_{T}T_{x}+h_{z}Z_{x}=h_{z}Z_{x}.
\label{306}
\end{equation}
Furthermore, we denote
\begin{equation*}
u(T,Z):=u(T,x(T,Z)),\ \ P(T,Z):=P(T,x(T,Z))\ \ \text{and}\ \ P_{x}(T,Z):=P_{x}(T,x(T,Z)).
\end{equation*}

In what follows, we define
\begin{equation}
w:=2\arctan u_{x}\ \ \text{and}\ \ v:=(1+u_{x}^{2})\frac{\partial x}{\partial Z}
\label{307}
\end{equation}
with $u_{x}=u_{x}(T,x(T,Z))$. By \eqref{307}, we can easily verify that
\begin{equation}
\frac{1}{1+u_{x}^{2}}=\cos^{2}\frac{w}{2},\ \ \frac{u_{x}^{2}}{1+u_{x}^{2}}=\sin^{2}\frac{w}{2},\ \
\frac{u_{x}}{1+u_{x}^{2}}=\frac{1}{2}\sin w,
\label{308}
\end{equation}
\begin{equation}
\frac{\partial x}{\partial Z}=\frac{v}{1+u_{x}^{2}}=v\cos^{2} \frac{w}{2}.
\label{309}
\end{equation}
By \eqref{309}, for any time $t=T$, we have
\begin{equation}
x(T,Z')-x(T,Z)=\int_{Z}^{Z'}\left(v\cos^{2} \frac{w}{2}\right)(T,s)ds.
\label{3010}
\end{equation}

Let $y=x(T,Z')$ and $x=x(T,Z)$. By using the identities \eqref{308}--\eqref{3010}, we get
\begin{align}
P(Z)
&=P(T,Z)\nonumber\\
&=\frac{1}{2}\int_{\mathbb{R}}e^{-|x(T,Z)-y|}\left(g(u)+\frac{f''(u)}{2}u_{x}^{2}
\right)(T,y)dy\nonumber\\
&=\frac{1}{2}\int_{\mathbb{R}}e^{-|\int_{Z}^{Z'}\left(v\cos^{2} \frac{w}{2}\right)(T,s)ds|}\left(g(u(Z'))\cos^{2} \frac{w(Z')}{2}+\frac{f''(u(Z'))}{2}\sin^{2} \frac{w(Z')}{2}\right)v(Z')dZ',
\label{3011}
\end{align}
and
\begin{align}
P_{x}(Z)
=&P_{x}(T,Z)\nonumber\\
=&\frac{1}{2}\left(\int_{x(T,Z)}^{+\infty}-\int_{-\infty}^{x(T,Z)}\right)e^{-|x(T,Z)-y|}\left(g(u)+\frac{f''(u)}{2}u_{x}^{2}
\right)(T,y)dy\nonumber\\
=&\frac{1}{2}\left(\int_{Z}^{+\infty}-\int_{-\infty}^{Z}\right)
e^{-|\int_{Z}^{Z'}\left(v\cos^{2} \frac{w}{2}\right)(T,s)ds|}\bigg(g(u(Z'))\cos^{2} \frac{w(Z')}{2}\nonumber\\
&+\frac{f''(u(Z'))}{2}\sin^{2} \frac{w(Z')}{2}\bigg)v(Z')dZ'.
\label{3012}
\end{align}

In what follows, we derive a closed semi-linear system for the unknowns $u, w$ and $v$ under the new variables $(T,Z)$. From \eqref{201} and \eqref{301}, we get
\begin{equation}
u_{T}(T,Z)=u_{t}(T,Z)+f'(u)u_{x}(T,Z)=-P_{x}(T,Z),
\label{3013}
\end{equation}
where $P_{x}(T,Z)$ is given at \eqref{3012}.

From \eqref{203} and \eqref{307}, we get
\begin{align}
w_{T}(T,Z)=&\frac{2}{1+u_{x}^{2}}\left(u_{tx}+f'(u)u_{xx}\right)(T,Z)\nonumber\\
=&\frac{2}{1+u_{x}^{2}}\left(-\frac{f''(u)}{2}u_{x}^{2}+g(u)-P\right)\nonumber\\
=&2\left(g(u)-P\right)\cos^{2}\frac{w}{2}-f''(u)\sin^{2}\frac{w}{2},
\label{3014}
\end{align}
where $P=P(T,Z)$ is given at \eqref{3011}.

Below, we will derive the equation for $v(T,Z)$. To this end, we need to use the following relation
\begin{equation}
Z_{tx}+f'(u)Z_{xx}=-f''(u)u_{x}Z_{x},
\label{3015}
\end{equation}
which can be derived from \eqref{304}. Then \eqref{203}, \eqref{305}, \eqref{307} and \eqref{3015} yield
\begin{align}
v_{T}(T,Z)=&\left(\frac{1+u_{x}^{2}}{Z_{x}}\right)_{T}\nonumber\\
=&\frac{Z_{x}(1+u_{x}^{2})_{T}-(1+u_{x}^{2})Z_{xT}}{Z^{2}_{x}}\nonumber\\
=&\frac{2u_{x}Z_{x}(u_{tx}+f'(u)u_{xx})
-(1+u_{x}^{2})\left(Z_{xt}+f'(u)Z_{xx}\right)}{Z^{2}_{x}}\nonumber\\
=&\frac{2u_{x}(u_{tx}+f'(u)u_{xx})
+f''(u)(1+u_{x}^{2})u_{x}}{Z_{x}}\nonumber\\
=&\frac{u_{x}}{Z_{x}}\left(2g(u)-2P+f''(u)\right)\nonumber\\
=&\left(g(u)-P+\frac{f''(u)}{2}\right)v\sin w.
\label{3016}
\end{align}

\section{Global solutions of the semi-linear system}

\setcounter{equation}{0}

\label{sec:4}

In this section, we prove the global existence of solutions for the semi-linear system.

Let initial datum $u_{0}(x)=\bar{u}\in H^{1}$ be given. We can transfer problem \eqref{201} into the following semi-linear system
\begin{align}
\begin{cases}
u_{T}=-P_{x},\\
w_{T}=2\left(g(u)-P\right)\cos^{2}\frac{w}{2}-f''(u)\sin^{2}\frac{w}{2},\\
v_{T}=\left(g(u)-P+\frac{f''(u)}{2}\right)v\sin w
\end{cases}
\label{401}
\end{align}
subject to the initial data
\begin{align}
\begin{cases}
u(0,Z)=\bar{u}(\bar{x}(Z)),\\
w(0,Z)=2\arctan \bar{u}_{x}(\bar{x}(Z)),\\
v(0,Z)=1,
\end{cases}
\label{402}
\end{align}
where $P$ and $P_{x}$ are given by \eqref{3011}--\eqref{3012}, respectively.

It is easy to verify that system \eqref{401} is invariant under translation by $2\pi$ in $w$. For simplicity, we choose $w\in[-\pi,\pi]$. We now consider system \eqref{401} as an  ordinary differential equations in the Banach space
\begin{equation}
X:=H^{1}(\mathbb{R})\times\left(L^{2}(\mathbb{R})\cap L^{\infty}(\mathbb{R})\right)\times L^{\infty}(\mathbb{R})
\label{403}
\end{equation}
with the norm
\begin{equation*}
\|(u,w,v)\|_{X}=\|u\|_{1}+\|w\|_{L^{2}}+\|w\|_{L^{\infty}}+\|v\|_{L^{\infty}}.
\end{equation*}

In what follows, we first prove the local existence of solutions for the Cauchy problem \eqref{401}--\eqref{402}. By the standard theory of ordinary differential equations in the Banach space, we only need to show that all functions on the right-hand side of system \eqref{401} are locally Lipschitz continuous. We then use the energy conservation property to extend the local solution to the global solution.

\begin{lemma}\label{lem401}
Let $\bar{u}\in H^{1}$. Then the Cauchy problem \eqref{401}--\eqref{402} has a unique solution defined on $[0,T]$ for some $T>0$.
\end{lemma}
\begin{proof}
To establish the local well-posedness, it suffices to prove the operator determined by the
right-hand side of \eqref{401}, which maps $(u,w,v)$ to
\begin{equation}
\left(-P_{x},\ \ 2(g(u)-P)\cos^{2}\frac{w}{2}-f''(u)\sin^{2}\frac{w}{2},\ \ \left(g(u)-P+\frac{f''(u)}{2}\right)v\sin w\right)
\label{404}
\end{equation}
is Lipschitz continuous on every bounded domain $\Omega \subset X$ of the following form
\begin{equation}
\Omega=\left\{(u,w,v):\|u\|_{1}\leq \kappa, \|w\|_{L^{2}}\leq \mu, \|w\|_{L^{\infty}}\leq \frac{3\pi}{2}, v(x)\in[v^{-},v^{+}]\
 \text{for a.e.}\ x\in \mathbb{R}\right\}
\label{405}
\end{equation}
for any positive constants $\kappa, \mu, v^{-}$ and $v^{+}$.

Since $f(u), g(u)\in C^{\infty}(\mathbb{R},\mathbb{R})$ and $g(0)=0$, and the uniform bounds on $w, v$, and the Sobolev inequality
\begin{equation}
\|u\|_{L^{\infty}}\leq \frac{1}{\sqrt{2}}\|u\|_{1},
\label{406}
\end{equation}
it is clear that the maps
\begin{equation}
2g(u)\cos^{2}\frac{w}{2},\ -f''(u)\sin^{2}\frac{w}{2}\ \ \text{and}\  \ \left(g(u)+\frac{f''(u)}{2}\right)v\sin w
\label{407}
\end{equation}
are all Lipschitz continuous as maps from $\Omega$ into $L^{2}\cap L^{\infty}$.
Therefore, we only need to prove the maps
\begin{equation}
(u,w,v)\mapsto (P,P_{x})
\label{408}
\end{equation}
are Lipschitz continuous from $\Omega$ into $L^{2}\cap L^{\infty}$. To this end, it suffices to show that the above maps are Lipschitz continuous from $\Omega$ into $H^{1}$.

In what follows, we derive some estimates for future. For $(u,w,v)\in\Omega$, we have
\begin{align}
measure\left\{Z\in \mathbb{R}:\left|\frac{w(Z)}{2}\right|\geq \frac{\pi}{4}\right\}
\leq& measure\left\{Z\in \mathbb{R}:\sin^{2}\frac{w(Z)}{2}\geq \frac{1}{4}\right\}\nonumber\\
\leq& 4\int_{\left\{Z\in \mathbb{R}:\sin^{2}\frac{w(Z)}{2}\geq \frac{1}{4}\right\}}\sin^{2}\frac{w(Z)}{2}dZ\nonumber\\
\leq& \int_{\left\{Z\in \mathbb{R}:\sin^{2}\frac{w}{2}\geq \frac{1}{4}\right\}}w^{2}(Z)dZ\nonumber\\
\leq& \mu^{2}.
\label{409}
\end{align}
Therefore, for any $Z_{1}<Z_{2}$, we get
\begin{equation}
\int_{Z_{1}}^{Z_{2}}v(s)\cos^{2}\frac{w(s)}{2}ds\geq \int_{\left\{s\in[Z_{1},Z_{2}],\left|\frac{w(s)}{2}\right|\leq \frac{\pi}{4}\right\}}\frac{v^{-}}{2}ds\geq\frac{v^{-}}{2}\left((Z_{2}-Z_{1})-\mu^{2}\right).
\label{4010}
\end{equation}
The inequality \eqref{4010} is a key estimate which guarantees that the exponential term in the formulate \eqref{3011}--\eqref{3012} for $P$ and $P_{x}$ decreasing quickly as $|Z-Z'|\rightarrow \infty$. Below, we introduce the exponentially decaying function
\begin{equation}
\Lambda(\eta):=\min \left\{1,e^{\left(\frac{\mu^{2}}{2}-\frac{|\eta|}{2}\right)v^{-}}\right\}.
\label{4011}
\end{equation}
Thus, we get
\begin{equation}
\|\Lambda(\eta)\|_{L^{1}}
=\left(\int_{|\eta|\leq \mu^{2}}+\int_{|\eta|\geq \mu^{2}}\right)\Lambda(\eta)d\eta
=2\mu^{2}+\frac{4}{v^{-}}.
\label{4012}
\end{equation}

In what follows, we prove that $P, P_{x}\in H^{1}$, namely,
\begin{equation}
P,\ \partial_{Z}P,\ P_{x},\ \partial_{Z}P_{x}\in L^{2}(\mathbb{R}).
\label{4013}
\end{equation}
It is obvious that the priori estimates for $P$ and $P_{x}$ are totally similar. For simplicity, we only consider the case for $P_{x}$.

From \eqref{3012}, we get
\begin{equation}
|P_{x}(Z)|\leq \frac{v^{+}}{2}\left|\Lambda\ast \left(g(u)\cos^{2} \frac{w}{2}+\frac{f''(u)}{2}\sin^{2} \frac{w}{2}\right)(Z)\right|.
\label{4014}
\end{equation}
Therefore, using the standard properties of convolutions, the Sobolev inequality 
and Young's inequality, we get
\begin{align}
\|P_{x}(Z)\|_{L^{2}}&\leq \frac{v^{+}}{2}\|\Lambda\|_{L^{1}}\left(\|g(u)\|_{L^{2}}
+\frac{|f''(u)|}{2}\|w^{2}\|_{L^{2}}\right)\nonumber\\
&\leq \frac{v^{+}}{2}\|\Lambda\|_{L^{1}}\left(C(\|u\|_{1})\|u\|_{L^{2}}
+\frac{|f''(u)|}{2}\|w\|_{L^{\infty}}\|w\|_{L^{2}}\right)\nonumber\\
&< \infty.
\label{4015}
\end{align}

Next, differentiating $P_{x}$ with respect to $Z$, we get
\begin{align}
\partial_{Z}P_{x}(Z)=&-\left(g(u(Z))\cos^{2} \frac{w(Z)}{2}+\frac{f''(u(Z))}{2}\sin^{2} \frac{w(Z)}{2}\right)v(Z)\nonumber\\
&+\frac{1}{2}\left(\int_{Z}^{+\infty}
-\int_{-\infty}^{Z}\right)e^{-\left|\int_{Z}^{Z'}(v\cos^{2} \frac{w}{2})(T,s)ds\right|}v(Z)\cos^{2} \frac{w(Z)}{2}\sign{Z'-Z}\nonumber\\
&\cdot \left(g(u(Z'))\cos^{2} \frac{w(Z')}{2}+\frac{f''(u(Z'))}{2}\sin^{2} \frac{w(Z')}{2}\right)v(Z')dZ'.
\label{4016}
\end{align}
Therefore,
\begin{align}
|\partial_{Z}P_{x}(Z)|\leq&v^{+}\left|g(u(Z))+\frac{|f''(u(Z))|}{8}v^{2}\right|\nonumber\\
&+\frac{(v^{+})^{2}}{2}\left|\Lambda\ast \left(g(u)\cos^{2} \frac{w}{2}+\frac{f''(u)}{2}\sin^{2} \frac{w}{2}\right)(Z)\right|.
\label{4017}
\end{align}
Furthermore, applying standard properties of convolutions and Young's inequality, we get
\begin{align}
\|\partial_{Z}P_{x}(Z)\|_{L^{2}}\leq& v^{+}\left(\|g(u)\|_{L^{2}}+\frac{|f''(u)|}{8}\|w^{2}\|_{L^{2}}\right)\nonumber\\
&+\frac{(v^{+})^{2}}{2}\|\Lambda\|_{L^{1}} \left(\|g(u)\|_{L^{2}}+\frac{|f''(u)|}{8}\|w^{2}\|_{L^{2}}\right)\nonumber\\
\leq& \left(v^{+}+\frac{(v^{+})^{2}}{2}\|\Lambda\|_{L^{1}}\right)\left(C(\|u\|_{1})
\|u\|_{L^{2}}+\frac{|f''(u)|}{8}\|w\|_{L^{\infty}}\|w\|_{L^{2}}\right)\nonumber\\
<& \infty.
\label{4018}
\end{align}
Thus, $P_{x}\in H^{1}(\mathbb{R})$. Note that the estimates for $P$ and $P_{x}$ can be obtained by the same method. Therefore, the proof of the relation \eqref{4013} is completed.

In what follows, we verify the Lipschitz continuity of the map given in \eqref{409}. This can be done by proving that for $(u,w,v)\in \Omega$, the partial derivatives
\begin{equation}
\frac{\partial P}{\partial u},\ \frac{\partial P}{\partial w},\ \frac{\partial P}{\partial v},\
\frac{\partial P_{x}}{\partial u},\ \frac{\partial P_{x}}{\partial w},\ \frac{\partial P_{x}}{\partial v}
\label{4019}
\end{equation}
are uniformly bounded linear operators from the appropriate spaces into $H^{1}$. Due to the fact that all the partial derivatives can be estimated by the same method, so we only detail the argument for $\frac{\partial P_{x}}{\partial u}$.

For every test function $\phi\in H^{1}$, the operators $\frac{\partial P_{x}}{\partial u}$ and $\frac{\partial (\partial_{Z}P_{x})}{\partial u}$ at a given point $(u,w,v)\in \Omega$ are defined by
\begin{align}
&\left(\frac{\partial P_{x}(u,w,v)}{\partial u}\cdot \phi\right)(Z)\nonumber\\
=&\frac{1}{2}\left(\int_{Z}^{+\infty}-\int_{-\infty}^{Z}\right)
e^{-\left|\int_{Z}^{Z'}(v\cos^{2} \frac{w}{2})(T,s)ds\right|}\left(g'(u(Z'))\cos^{2} \frac{w(Z')}{2}+\frac{f'''(u(Z'))}{2}\sin^{2} \frac{w(Z')}{2}\right)\nonumber\\
&\cdot v(Z')\phi(Z') dZ'
\label{4020}
\end{align}
and
\begin{align}
&\left(\frac{\partial (\partial_{Z}P_{x})(u,w,v)}{\partial u}\cdot \phi\right)(Z)\nonumber\\
=&-\left(g'(u(Z))\cos^{2} \frac{w}{2}+\frac{f'''(u(Z))}{2}\sin^{2} \frac{w}{2}\right)v(Z)\phi(Z)\nonumber\\
&+\frac{1}{2}\left(\int_{Z}^{+\infty}-\int_{-\infty}^{Z}\right)
e^{-\left|\int_{Z}^{Z'}(v\cos^{2} \frac{w}{2}ds)(T,s)\right|}v(Z)\cos^{2} \frac{w(Z)}{2}\sign{Z'-Z}\nonumber\\
&\cdot \left(g'(u(Z'))\cos^{2} \frac{w(Z')}{2}+\frac{f'''(u(Z'))}{2}\sin^{2} \frac{w(Z')}{2}\right)v(Z')\phi(Z') dZ'.
\label{4026}
\end{align}
Therefore, we obtain
\begin{align}
\left\|\frac{\partial P_{x}}{\partial u}\cdot \phi\right\|_{L^{2}}
\leq&v^{+}\left\|\Lambda\ast \left(g'(u)+\frac{f'''(u)}{2}\right)\right\|_{L^{2}}\|\phi\|_{L^{\infty}}\nonumber\\
\leq& Cv^{+}\|\Lambda\|_{L^{1}} \|u\|_{L^{2}}\|\phi\|_{1}\nonumber\\
\leq& Cv^{+}\|\Lambda\|_{L^{1}} \|u\|_{1}\|\phi\|_{1}\nonumber\\
<& \infty
\label{4022}
\end{align}
and
\begin{align}
&\left\|\frac{\partial(\partial_{Z} P_{x})}{\partial u}\cdot \phi\right\|_{L^{2}}\nonumber\\
\leq& v^{+}\left\|g'(u)+\frac{f'''(u)}{2}\right\|_{L^{2}}\|\phi\|_{L^{\infty}}
+\frac{(v^{+})^{2}}{2}\left\|\Lambda\ast \left(g'(u)+\frac{f'''(u)}{2}\right)\right\|_{L^{2}}\|\phi\|_{L^{\infty}}\nonumber\\
\leq&\left(v^{+}+\frac{(v^{+})^{2}}{2}\|\Lambda\|_{L^{1}}\right) \left(\|g'(u)+\frac{f'''(u)}{2}\|_{L^{2}}\right)\|\phi\|_{1}\nonumber\\
<& \infty,
\label{4023}
\end{align}
where we used the facts that
$$\|u\|_{L^{\infty}}\leq \|u\|_{1},\ \
\|\phi\|_{L^{\infty}}\leq \|\phi\|_{1}
\ \text{and}\ |g'(u)|\leq C(\|u\|_{1})|u|.$$
Hence we obtain that $\frac{\partial P_{x}}{\partial u}$ is a bounded linear operator from $H^{1}$ into $H^{1}$. As above, we can obtain the boundedness of other partial derivatives, thus the uniform Lipschitz continuous of the map in \eqref{408} is verified. Then using the standard ODE theory in the Banach space, we can establish the local existence of solutions for the Cauchy problem \eqref{401}--\eqref{402}, namely, the Cauchy problem \eqref{401}--\eqref{402} has a unique solution on $[0,T]$ for some $T>0$.
This completes the proof of Lemma \ref{lem401}.
\end{proof}

In what follows, we extend the local solution obtained in Lemma \ref{lem401} globally. To this end, it suffices to prove that for all $T<\infty$,
\begin{equation}
\|u\|_{1}+\|w\|_{L^{2}}+\|w\|_{L^{\infty}}
+\|v\|_{L^{\infty}}+\left\|\frac{1}{v}\right\|_{L^{\infty}}<\infty.
\label{4024}
\end{equation}

\begin{lemma}\label{lem402}
Let $\bar{u}\in H^{1}$. Then the Cauchy problem \eqref{401}--\eqref{402} has a unique solution defined for all $T>0$.
\end{lemma}
\begin{proof}
For the local solution obtained in Lemma \ref{lem401}, we claim that
\begin{equation}
u_{Z}=\frac{u_{x}}{Z_{x}}
=\frac{u_{x}}{1+u_{x}^{2}}v
=\frac{1}{2}v\sin w.
\label{4025}
\end{equation}
Indeed, from \eqref{401}, we have
\begin{align}
u_{ZT}=u_{TZ}=&-\partial_{Z}P_{x}\nonumber\\
=&\left((g(u)-P(Z))\cos^{2} \frac{w}{2}+\frac{f''(u)}{2}\sin^{2}\frac{w}{2}\right)v(Z).
\label{4026}
\end{align}
On the other hand,
\begin{align}
&\left(\frac{1}{2}v\sin w\right)_{T}\nonumber\\
=&\frac{1}{2}v_{T}\sin w+\frac{1}{2}vw_{T}\cos w\nonumber\\
=&\frac{1}{2}v\sin^{2} w\left(g(u)-P(Z)+\frac{f''(u)}{2}\right)
+\frac{v}{2}\cos w\left(2(g(u)-P(Z))\cos^{2} \frac{w}{2}-f''(u)\sin^{2} \frac{w}{2}\right)\nonumber\\
=&\left((g(u)-P(Z))\cos^{2} \frac{w}{2}+\frac{f''(u)}{2}\sin^{2}\frac{w}{2}\right)v(Z).
\label{4027}
\end{align}
Applying the initial data, we know
\begin{equation}
u_{Z}=\frac{1}{2}\sin w\ \ \text{and}
\ \ v=1,\ \ as\ \ T=0,
\label{4028}
\end{equation}
which means that \eqref{4025} holds initially. Thus, we infer that \eqref{4025} remains valid for all $T$ as long as the solution exists.

In what follows, we verify the boundedness of \eqref{4024}. To this end, we check the conservation law $E(t)$. In the new system \eqref{401}--\eqref{402}, the conservation law of $E(T)$ read
\begin{equation}
E(T)=\int_{\mathbb{R}}\left(u^{2}\cos^{2}\frac{w}{2}
+\sin^{2}\frac{w}{2}\right)v(T,Z)dZ=\bar{E}(0).
\label{4029}
\end{equation}
To prove \eqref{4029}, it is useful to give the following identities in terms of the $Z$-derivatives.
\begin{equation}
P_{Z}=v(Z) P_{x}(Z)\cos^{2} \frac{w(Z)}{2}
\label{4030}
\end{equation}
and
\begin{equation}
\partial_{Z}P_{x}=-\left((g(u)-P(Z))\cos^{2} \frac{w}{2}+\frac{f''(u)}{2}\sin^{2}\frac{w}{2}\right)v(Z).
\label{4031}
\end{equation}

Applying \eqref{401}, \eqref{4030} and \eqref{4031}, a direct calculation reveals that
\begin{align}
\frac{dE(T)}{dT}=&\int_{\mathbb{R}}\left(\left(u^{2}\cos^{2}\frac{w}{2}
+\sin^{2}\frac{w}{2}\right)v(T,Z)\right)_{T}dZ\nonumber\\
=&\int_{\mathbb{R}}\bigg(\left(2uu_{T}\cos^{2}\frac{w}{2}
-u^{2}w_{T}\cos\frac{w}{2}\sin\frac{w}{2}
+w_{T}\sin \frac{w}{2}\cos \frac{w}{2}\right)v\nonumber\\
&+\left(u^{2}\cos^{2}\frac{w}{2}+\sin^{2}\frac{w}{2}\right)v_{T}\bigg)dZ\nonumber\\
=&\int_{\mathbb{R}}\Bigg\{\bigg(-2uP_{x}\cos^{2}\frac{w}{2}-2u^{2}(g(u)-P(Z))\sin \frac{w}{2}\cos^{3} \frac{w}{2}\nonumber\\
&+f''(u)u^{2}\sin^{3} \frac{w}{2}\cos \frac{w}{2}
+2(g(u)-P(Z))\sin \frac{w}{2}\cos^{3} \frac{w}{2}-f''\sin^{3} \frac{w}{2}\cos \frac{w}{2}\bigg)v\nonumber\\
&+\left(u^{2}\cos^{2}\frac{w}{2}
+\sin^{2}\frac{w}{2}\right)\left(g(u)-P(Z)+\frac{f''(u)}{2}\right)v\sin w\Bigg\}dZ\nonumber\\
=&\int_{\mathbb{R}}\left(-2uP_{x}\cos^{2}\frac{w}{2}+\frac{f''(u)}{2}u^{2}\sin w
+(g(u)-P(Z))\sin w\right)v(T,Z)dZ.
\label{4032}
\end{align}
In view of \eqref{4025} and \eqref{4030}, we have
\begin{equation}
(uP)_{Z}=u_{Z}P+uP_{Z}=vP\sin \frac{w}{2}\cos \frac{w}{2}+uvP_{x}\cos^{2} \frac{w}{2}
\label{4033}
\end{equation}
and
\begin{equation}
g(u)v\sin w=g(u)u_{Z}=(G(u))_{Z},
\label{4034}
\end{equation}
where $G(u)=\int_{0}^{u}g(s)ds$.

On the other hand,
\begin{align}
\frac{f''(u)}{2}vu^{2}\sin w
=&f''(u)u^{2}u_{Z}\nonumber\\
=&\left(f'(u)u^{2}\right)_{Z}-2f'(u)uu_{Z}\nonumber\\
=&\left(f'(u)u^{2}\right)_{Z}-2\left((f(u)u)_{Z}-f(u)u_{Z}\right)\nonumber\\
=&\left(f'(u)u^{2}\right)_{Z}-2(f(u)u)_{Z}+2f(u)u_{Z}\nonumber\\
=&\left(f'(u)u^{2}\right)_{Z}-2(f(u)u)_{Z}+2(F(u))_{Z},
\label{4035}
\end{align}
where $F(u)=\int_{0}^{u}f(s)ds$.

Therefore,
\begin{align}
\frac{dE(T)}{dT}=&\int_{\mathbb{R}}\left(\left(u^{2}\cos^{2}\frac{w}{2}
+\sin^{2}\frac{w}{2}\right)v(T,Z)\right)_{T}dZ\nonumber\\
=&\int_{\mathbb{R}}\left(G(u)-2uP+f'(u)u^{2}-2f(u)u+2F(u)\right)_{Z}dZ\nonumber\\
=&0,
\label{4036}
\end{align}
where in deriving the last equality we have used the asymptotic property
\begin{equation*}
\lim_{|Z|\rightarrow \infty}u(Z)=0\ \ \text{as}\ \ u\in H^{1}(\mathbb{R}),
\end{equation*}
and the fact that $P(Z)$ is uniformly bounded. This proves \eqref{4029}.

We have now proved the conservation law \eqref{4029} in the new variables along any solution of \eqref{401}--\eqref{402}. In what follows, we use the conservation law \eqref{4029} to derive a priori estimate on $\|u(T)\|_{L^{\infty}}$. It is clear that
\begin{align}
\sup_{Z\in\mathbb{R}}\left|u^{2}(T,Z)\right|
\leq 2\int_{\mathbb{R}}\left|uu_{Z}\right|dZ
&\leq 2\int_{\mathbb{R}}\left|u\sin \frac{w}{2} \cos \frac{w}{2}\right|v dZ\nonumber\\
&\leq \int_{\mathbb{R}} \left|\sin^{2} \frac{w}{2}+u^{2}\cos^{2} \frac{w}{2}\right|vdZ\nonumber\\
&\leq \bar{E}(0).
\label{4037}
\end{align}

From \eqref{4029} and \eqref{3011}, we can easily verify
\begin{align}
\|P(T)\|_{L^{\infty}}\leq& \frac{1}{2}\left\|e^{-|x|}\right\|_{L^{\infty}}
\left\|g(u)\right\|_{L^{2}}
+\frac{1}{2}\left\|e^{-|x|}\right\|_{L^{\infty}}
\left\|\frac{f''(u)}{2}u_{x}^{2}\right\|_{L^{1}}
\leq C\bar{E}(0),\nonumber\\
\|P_{x}(T)\|_{L^{\infty}}\leq& C\bar{E}(0).
\label{4038}
\end{align}
Hence we recover the estimate \eqref{2010} in the new variables.

Below, using the estimates \eqref{4037}, \eqref{4038} and the third equation in system \eqref{401}, we can prove the  $L^{\infty}$ bound for $v(T,Z)$. Indeed, we have
\begin{equation}
|v_{T}(T,Z)|\leq C\bar{E}(0)v(T,Z).
\label{4039}
\end{equation}
Since $v(0,Z)=1$, \eqref{4039} yields
\begin{equation}
e^{-C\bar{E}(0)T}\leq v(T,Z)\leq e^{C\bar{E}(0)T}.
\label{4040}
\end{equation}
Similarly, it follows from the second equation of system \eqref{401} that
\begin{equation}
|w_{T}(T,Z)|\leq C,
\label{4041}
\end{equation}
where $C=C(\bar{E}(0))>0$. Consequently,
\begin{equation}
\|w(T,Z)\|_{L^{\infty}}\leq \|w(0,Z)\|_{L^{\infty}}+CT.
\label{4042}
\end{equation}

In what follows, we prove that $\|u\|_{1}$ is bounded for any bounded internal of time $T$. To this end, multiplying $2u$ to the first equation of system \eqref{401}, we get
\begin{equation}
\frac{d}{dT}\|u(T)\|_{L^{2}}^{2}\leq 2\|u(T)\|_{L^{\infty}}\|P_{x}(T)\|_{L^{1}}.
\label{4043}
\end{equation}
Differentiating the first equation of system \eqref{401} with respect to $Z$, we get
\begin{equation}
u_{TZ}(T,Z)=-\partial_{Z}P_{x}(T,Z).
\label{4044}
\end{equation}
Multiplying \eqref{4044} by $2u_{Z}$, we obtain
\begin{equation}
\frac{d}{dT}\|u_{Z}(T)\|_{L^{2}}^{2}\leq 2\|u_{Z}(T)\|_{L^{\infty}}\|\partial_{Z}P_{x}(T)\|_{L^{1}}.
\label{4045}
\end{equation}

On the other hand, it is known from \eqref{4025} that
\begin{equation}
\|u_{Z}(T)\|_{L^{\infty}}\leq \frac{1}{2}\|v(T)\|_{L^{\infty}}\leq \frac{1}{2}e^{C\bar{E}(0)T}.
\label{4046}
\end{equation}
In order to prove that $\|u\|_{1}$ is bounded for any $T<\infty$, it suffices to show $\|P_{x}(T)\|_{L^{1}}$ and $\|\partial_{Z}P_{x}(T)\|_{L^{1}}$ are bounded. It is observed that these two terms can be estimated by the similar method, we only need to consider
$\|\partial_{Z}P_{x}(T)\|_{L^{1}}$.

Indeed, for $Z<Z'$, we have
\begin{align}
\int^{Z'}_{Z}(v\cos^{2} \frac{w}{2})(T,s)ds
\geq&\int_{\big\{s\in[Z,Z'],\left|\frac{w}{2}\right|\leq\frac{\pi}{4}\big\}}(v\cos^{2} \frac{w}{2})(s)ds\nonumber\\
\geq&\int_{\big\{s\in[Z,Z'],\left|\frac{w}{2}\right|\leq\frac{\pi}{4}\big\}}\frac{v(s)}{2}ds\nonumber\\
\geq&\frac{v^{-}}{2}(Z'-Z)
-\int_{\big\{s\in[Z,Z'],\left|\frac{w}{2}\right|\geq\frac{\pi}{4}\big\}}\frac{v(s)}{2}ds\nonumber\\
\geq&\frac{v^{-}}{2}(Z'-Z)
-\int_{\big\{s\in[Z,Z'],\left|\frac{w}{2}\right|\geq\frac{\pi}{4}\big\}}v(s)\sin^{2}\frac{w(s)}{2}ds
\nonumber\\
\geq&\frac{v^{-}}{2}(Z'-Z)-\bar{E}(0),
\label{4047}
\end{align}
where $v^{-}=e^{-C\bar{E}(0)T}$.

Below, we introduce the exponentially decaying function
\begin{equation}
\Gamma(\eta):=\min\left\{1,e^{\bar{E}(0)-\frac{v^{-}|\eta|}{2}}\right\}
\label{4048}
\end{equation}
with
\begin{equation}
\|\Gamma(\eta)\|_{L^{1}}=\frac{4(\bar{E}(0)+1)}{v^{-}}=4e^{C\bar{E}(0)T}(\bar{E}(0)+1).
\label{4049}
\end{equation}
Therefore, from \eqref{4031} and \eqref{4049}, we get
\begin{align}
\|\partial_{Z}P_{x}\|_{L^{1}}
=&\left\|-\left(g(u)\cos^{2}\frac{w}{2}+\frac{f''(u)}{2}\sin^{2}\frac{w}{2}\right)v
+vP\cos^{2}\frac{w}{2}\right\|_{L^{1}}\nonumber\\
\leq&C\bar{E}(0)+\frac{v^{+}}{2}\left\|\Gamma\ast
\left(g(u)\cos^{2}\frac{w}{2}+\frac{f''(u)}{2}\sin^{2}\frac{w}{2}\right)v\right\|_{L^{1}}\nonumber\\
\leq&C\bar{E}(0)+\frac{v^{+}}{2}\|\Gamma\|_{L^{1}}
\left\|\left(g(u)\cos^{2}\frac{w}{2}+\frac{f''(u)}{2}\sin^{2}\frac{w}{2}\right)v\right\|_{L^{\infty}}\nonumber\\
\leq&C\bar{E}(0)+Cv^{+}e^{C\bar{E}(0)T}(\bar{E}(0)+1)\bar{E}(0)\nonumber\\
<&\infty.
\label{4050}
\end{align}

It then turns out that $\|u\|_{1}$ is bounded on the bounded internals of time $T$. Finally, multiplying $2w$ to the second equation of system \eqref{401}, we get
\begin{align}
\frac{d}{dT}\|w(T)\|_{L^{2}}^{2}
\leq& 4\int_{\mathbb{R}}|(g(u)-P)w|dz+\frac{|f''(u)|}{2}\int_{\mathbb{R}}|w^{3}|dZ\nonumber\\
\leq& C(\|u\|_{L^{2}}+\|P\|_{L^{2}})
\|w\|_{L^{\infty}}+\frac{|f''(u)|}{2}\|w\|_{L^{\infty}}\|w\|_{L^{2}}^{2}.
\label{4051}
\end{align}
By the previous bounds, it is clear that $\|w\|_{L^{2}}$ remains bounded on bounded internals of time $T$. This completes the proof that the local solution of system \eqref{401} can be extended globally in time.
\end{proof}

Furthermore, similar to the result in \cite{Bressan2007}, we have the following property for the global solution in Lemma \eqref{lem402}.
\begin{lemma}\label{lem403}
Consider the set of time
\begin{equation*}
\Theta:=\{T\geq 0,\ measure\{Z\in \mathbb{R}:w(T,Z)=-\pi\}>0\}.
\end{equation*}
Then
\begin{equation}
measure(\Theta)=0.
\label{4052}
\end{equation}
\end{lemma}

\section{Solutions to the nonlinear dispersive wave equations}
\setcounter{equation}{0}

\label{sec:5}

In this section, we construct the weak solution to \eqref{201} by using an inverse translation on the solution of system \eqref{401}.

We define $t$ and $x$ as functions of $T$ and $Z$ by
\begin{equation}
x(T,Z)=\bar{x}(Z)+\int_{0}^{T}f'(u(\zeta,Z))d\zeta,\ \ t=T.
\label{501}
\end{equation}
Thus the above function $x(T,Z)$ provides a solution to the following initial problem
\begin{equation}
\frac{\partial x(T,Z)}{\partial T}=f'(u(T,Z)),\ \ x(0,Z)=\bar{x}(Z),
\label{502}
\end{equation}
which means that $x(T,Z)$ is a characteristic.

In what follows, we will prove that the functions
\begin{equation}
u(t,x)=u(T,Z),\ \ \text{if}\ \ t=T,\ x=x(T,Z),
\label{503}
\end{equation}
provides a weak solution of \eqref{201}.

\begin{definition}\label{def501}
The energy conservative solution $u(t,x)$ of the Cauchy problem \eqref{201} with the initial datum $\bar{u}(x)$ has the following properties.\\
(i). The map $t\rightarrow u(t)$ is Lipschitz continuous from $\mathbb{R}$ into $L^{2}(\mathbb{R})$ with $u(t,\cdot)\in H^{1}(\mathbb{R})$ for all $t\geq 0$.\\
(ii). The solution $u=u(t,x)$ satisfies the initial datum $\bar{u}(x)\in H^{1}(\mathbb{R})$ and
\begin{align}
\int\int_{\Delta}\left(-u_{x}\left(\psi_{t}+f'(u)\psi_{x}\right)
+\psi\left(P(Z)-g(u)
-\frac{f''(u)}{2}u_{x}^{2}\right)\right)dxdt&\nonumber\\
-\int_{\mathbb{R}}u_{x}(0,x)\psi(0,x)dx&=0,
\label{5}
\end{align}
for any text function $\psi\in C_{c}^{1}(\Delta)$, where $\Delta=\{(t,x):(t,x)\in \mathbb{R}_{+}\times \mathbb{R}\}$.
\end{definition}

In what follows, we state the main result on the global well-posedness of the energy conservative solution for \eqref{201}.

\begin{theorem}\label{the501}
Let the initial datum $\bar{u}(x)\in H^{1}(\mathbb{R})$. Then the Cauchy problem \eqref{201} with the initial datum $\bar{u}(x)$ has a global energy conservative solution $u(t,x)$ in the sense of Definition \ref{def501}. Furthermore, the solution $u(t,x)$ satisfies the following properties:\\
(i) $u(t,x)$ is uniformly H\"older continuous with exponent $\frac{1}{2}$ on both $t$ and $x$.\\
(ii) The energy $u^{2}+u_{x}^{2}$ is almost conserved, i.e.,
\begin{equation}
\|u\|_{1}^{2}=\|\bar{u}\|_{1}^{2},\ \ \text{for\ a.e.}\ \ t\in \mathbb{R}_{+}.
\label{504}
\end{equation}
(iii) The solution $u(t,x)$ is continuously depending on the initial datum $\bar{u}(x)$. That is, let $\bar{u}_{n}$ be a sequence of initial datum such that
\begin{equation*}
\|\bar{u}_{n}-\bar{u}\|_{1}\rightarrow 0\ \ \text{as}\ \ n\rightarrow \infty.
\end{equation*}
Then the corresponding solutions $u_{n}(t,x)$ converges to $u(t,x)$ uniformly for $(t,x)$ in any bounded sets.
\end{theorem}

\begin{proof}
The argument of proof is divided into seven steps.

\text{Step 1.} We show that the continuous map $(T,Z)\rightarrow (t,x(T,Z))$
is a surjective function in $\mathbb{R}^{2}$. Indeed, by \eqref{403} and \eqref{501}, we get
\begin{equation*}
\bar{x}(Z)-\sqrt{E(0)}T\leq x(T,Z)\leq \bar{x}(Z)+\sqrt{E(0)}T.
\end{equation*}
Then from \eqref{303}, we deduce that
\begin{equation*}
\lim_{Z\rightarrow \pm\infty}x(T,Z)=\pm\infty.
\end{equation*}
Therefore, the image of continuous map $(T,Z)\rightarrow (t,x(T,Z))$ covers the entire plane $\mathbb{R}^{2}$.

\text{Step 2.} We claim that
\begin{equation}
x_{Z}=v\cos^{2} \frac{w}{2}\ \ \text{for}\ \ T\geq 0,\ \ \text{a.e.}\ \ Z\in \mathbb{R}.
\label{505}
\end{equation}
In fact, from \eqref{401} and \eqref{4025}, we get
\begin{align}
\left(v\cos^{2} \frac{w}{2}\right)_{T}
=&-w_{T}\sin \frac{w}{2}\cos \frac{w}{2}+v_{T}\cos^{2} \frac{w}{2}\nonumber\\
=&-v\sin \frac{w}{2}\cos \frac{w}{2}\left(2(g(u)-P(Z))\cos^{2} \frac{w}{2}
-f''(u)\sin^{2} \frac{w}{2}\right)\nonumber\\
&+v\sin w\cos^{2} \frac{w}{2}\left(g(u)-P(Z)+\frac{f''(u)}{2}\right)\nonumber\\
=&\frac{f''(u)}{2}v\sin w\nonumber\\
=&\left(f'(u)\right)_{Z}.
\label{506}
\end{align}

On the other hand, \eqref{502} implies
\begin{equation}
\frac{\partial x_{Z}}{\partial T}=\left(f'(u)\right)_{Z}.
\label{507}
\end{equation}
Since the function $x\rightarrow 2\arctan \bar{u}_{x}(x)$ is measure, the identity \eqref{505} holds for almost every $Z\in \mathbb{R}$ and $T=0$. By the above computations, it remains true for all times $T\geq 0$. Furthermore, the function $x(T,Z)$ is non-decreasing on $Z$ when $T$ is fixed.

\text{Step 3.} Our goal is to show that $u(t,x)=u(T,x(T,Z))$ is well defined. In fact,
if $x(T,Z_{1})=x(T,Z_{2})$ for $Z_{1}<Z_{2}$, then we have
\begin{equation*}
x(T,Z)=x(T,Z_{1})\ \text{for}\ Z\in[Z_{1},Z_{2}],
\end{equation*}
where we use the non-decreasing property of $x(T,Z)$ on $Z$. From \eqref{505}, we get
\begin{equation*}
\cos \frac{w(T,Z)}{2}=0\ \text{for}\ Z\in[Z_{1},Z_{2}].
\end{equation*}
Therefore,
\begin{equation*}
u(T,Z_{2})-u(T,Z_{1})=\int_{Z_{1}}^{Z_{2}}\frac{v}{2}\sin wds=0.
\end{equation*}
This proves $u(t,x)\rightarrow u(T,x(T,Z))$ is well defined for all $t\geq 0$ and $x\in \mathbb{R}$.

\text{Step 4.} We discuss the regularity of $u(t,x)$ and energy conservation. By \eqref{4012}, we know that $E(T)$ is conservative on $(T,Z)$ coordinates.

For any given time $t$, we have
\begin{align}
E(0)=\bar{E}(0)=E(T)
=&\int_{\mathbb{R}}\left(u^{2}\cos^{2}\frac{w}{2}+\sin^{2}\frac{w}{2}\right)vdZ\nonumber\\
\geq&\int_{\{\cos>-1\}}\left(u^{2}\cos^{2}\frac{w}{2}+\sin^{2}\frac{w}{2}\right)vdZ\nonumber\\
\geq&\int_{\{\cos>-1\}}\left(u^{2}+u_{x}^{2}\right)vdZ\nonumber\\
=&E(t).
\label{508}
\end{align}
By Lemma \ref{lem403}, we obtain that \eqref{504} holds for almost all $t$.

Applying the Sobolev inequality, we obtain that \eqref{508} implies the uniform H\"older continuity with the exponent $\frac{1}{2}$ for $u(t,x)$ as a function of $x$. By the first equation in system \eqref{401} and the boundedness of $\|P_{x}(t)\|_{L^{\infty}}$, we can infer $u(t,x(t))$ is H\"older continuity with the exponent $\frac{1}{2}$. In fact,
\begin{align*}
&|u(t,x)-u(s,x)|\nonumber\\
\leq&|u(t,x)-u(t,y)|+|u(t,y)-u(s,y(s,Z))|+|u(s,y(s,Z))-u(s,x(t,Z))|\nonumber\\
\leq&\sqrt{E(0)}|x-y|^{\frac{1}{2}}+\sqrt{E(0)}|y(s,Z)-x(t,Z)|^{\frac{1}{2}}
+\int_{s}^{t}|P_{x}(\theta,Z)|d\theta\nonumber\\
\leq&C\left(|x-y|^{\frac{1}{2}}+|t-s|^{\frac{1}{2}}+|t-s|\right),
\end{align*}
where we choose $Z\in \mathbb{R}$ such that the characteristic $t\mapsto x(t)$ passes through the point $(s,y)$. This implies that $u(t,x)$ is uniform H\"older continuity with the exponent $\frac{1}{2}$ on $t$ and $x$.

Step 5. We prove that the map $t\rightarrow u(t,\cdot)$ is Lipschitz continuous on $L^{2}(\mathbb{R})$ norm. In fact, consider now any time interval $[\theta,\theta+\varepsilon]$. For any given point $(\theta,\tilde{x})$, we have the characteristic curve $T\rightarrow x(T,Z)$ passing through $(\theta,\tilde{x})$, i.e.,
$x(\theta)=\tilde{x}$.

By \eqref{401} and \eqref{4027}, we get
\begin{align}
&\left|u(\theta+\varepsilon,\tilde{x})-u(\theta,\tilde{x})\right|\nonumber\\
\leq&\left|u(\theta+\varepsilon,\tilde{x})-u(\theta+\varepsilon,x(\theta+\varepsilon,Z)\right|
+\left|u(\theta+\varepsilon,x(\theta+\varepsilon,Z)-u(\theta,\tilde{x})\right|\nonumber\\
\leq&\sup_{|\xi-\widetilde{x}|\leq \sqrt{E(0)}\varepsilon}
\left|u(\theta+\varepsilon,\xi)-u(\theta+\varepsilon,\widetilde{x})\right|
+\int_{\theta}^{\theta+\varepsilon}|P_{x}|dt.
\label{509}
\end{align}
Integrating \eqref{509} with respect to $x$ over $\mathbb{R}$, and using the boundedness of $\|u_{x}\|_{L^{2}}$ and $\|P_{x}\|_{L^{2}}$, we get
\begin{align}
&\int_{\mathbb{R}}\left|u(\theta+\varepsilon,\tilde{x})-u(\theta,\tilde{x})\right|^{2}dx\nonumber\\
\leq&2\int_{\mathbb{R}}\left(\int_{\tilde{x}-\sqrt{E(0)}\varepsilon}^{{\tilde{x}+\sqrt{E(0)}\varepsilon}}
|u_{x}(\theta+\varepsilon,\xi)|d\xi\right)^{2}dx\nonumber\\
&+2\int_{\mathbb{R}}\left(\int_{\theta}^{\theta+\varepsilon}|P_{x}(t,Z)|dt\right)^{2}
v(t,Z)\cos^{2}\frac{w(t,Z)}{2}dZ\nonumber\\
\leq&4\sqrt{\bar{E}(0)}\varepsilon\int_{\mathbb{R}}\left(\int_{\tilde{x}-\sqrt{E(0)}\varepsilon}
^{{\tilde{x}+\sqrt{E(0)}\varepsilon}}
|u_{x}(\theta+\varepsilon,\xi)|^{2}d\xi\right)dx\nonumber\\
&+2\varepsilon\|v(\theta,Z)\|_{L^{\infty}}\int_{\mathbb{R}}\left(\int_{\theta}^{\theta+\varepsilon}
|P_{x}(t,Z)|^{2}dt\right)dZ\nonumber\\
\leq&8\bar{E}(0)\varepsilon^{2}\|u_{x}(\theta+\varepsilon)\|_{L^{2}}^{2}
+2\varepsilon\|v(\theta)\|_{L^{\infty}}\int_{\theta}^{\theta+\varepsilon}
\|P_{x}(t)\|_{L^{2}}^{2}dt\nonumber\\
\leq& C\varepsilon^{2}.
\label{5010}
\end{align}
Consequently, this completes the proof of the Lipschitz continuity to the map $t\mapsto u(t)$ on $L^{2}(\mathbb{R})$-norm.

Step 6. We verify that the identity \eqref{5} holds for any test function $\psi\in C_{c}^{1}(\Delta)$, which also implies that the function $u$ provides a weak solution of \eqref{201}. To see this, we define
\begin{equation}
\Delta=\{(T,Z):(T,Z)\in \mathbb{R}_{+}\times \mathbb{R}\}\ \ \text{and}\ \
\bar{\Delta}=\Delta\cap\left\{(T,Z):\cos\frac{w(T,Z)}{2}\neq 0\right\}.
\label{5011}
\end{equation}

From \eqref{4026} and \eqref{4052}, we get
\begin{align}
0=&\int\int_{\Delta}\left(u_{ZT}\psi+\psi\left((P(Z)-g(u))\cos^{2}\frac{w}{2}
-\frac{f''(u)}{2}\sin^{2}\frac{w}{2}\right)v\right)dZdT\nonumber\\
=&\int\int_{\Delta}\left(-u_{Z}\psi_{T}+\psi\left((P(Z)-g(u))\cos^{2}\frac{w}{2}
-\frac{f''(u)}{2}\sin^{2}\frac{w}{2}\right)v\right)dZdT\nonumber\\
&-\int_{\mathbb{R}}u_{Z}(0,Z)\psi(0,Z)dZ\nonumber\\
=&\int\int_{\bar{\Delta}}\left(-u_{Z}\psi_{T}+\psi\left((P(Z)-g(u))\cos^{2}\frac{w}{2}
-\frac{f''(u)}{2}\sin^{2}\frac{w}{2}\right)v\right)dZdT\nonumber\\
&-\int_{\mathbb{R}}u_{Z}(0,Z)\psi(0,Z)dZ\nonumber\\
=&\int\int_{\Delta}\left(-u_{x}\left(\psi_{t}+f'(u)\psi_{x}\right)
+\psi\left(P(Z)-g(u)
-\frac{f''(u)}{2}u_{x}^{2}\right)\right)dxdt\nonumber\\
&-\int_{\mathbb{R}}u_{x}(0,x)\psi(0,x)dx,
\label{5012}
\end{align}
which implies that $u(t,x)$ is a global solution of \eqref{201} in the sence of Definition \ref{def501}.
Furthermore, from Lemma \ref{lem403}, \eqref{4029} and \eqref{508}, we obtain the identity in \eqref{504}.

Step 7. The continuous dependence result can be directly obtained by following the argument in \cite{Bressan2007}.
\end{proof}


\bibliographystyle{elsarticle-num}
\bibliography{<your-bib-database>}



{\bf Acknowledgement.}
This work is partially supported by NSFC Grants (nos. 12225103, 12071065 and 11871140) and the National Key Research and Development Program of China (nos. 2020YFA0713602 and 2020YFC1808301) and Natural Science Foundation of Gansu Province (no. 21JR7RA552).

\end{document}